\documentclass{article}
\usepackage{amsmath,amsthm,amssymb}

\newcommand{\refpart}[1]{{\it (#1)}}  

\renewcommand\ge\geqslant
\renewcommand\geq\geqslant
\renewcommand\le\leqslant
\renewcommand\leq\leqslant

\newtheorem{theorem}{Theorem}[section]

\newtheorem{conjecture}[theorem]{Conjecture}

\newtheorem{lemma}[theorem]{Lemma}
\newtheorem{proposition}[theorem]{Proposition}

\newtheorem{remark}[theorem]{Remark}

\newcommand{\CC}{\mathbb{C}}
\newcommand{\RR}{\mathbb{R}}
\newcommand{\QQ}{\mathbb{Q}}

\newcommand{\NN}{{\mathcal N}}
\newcommand{\WW}{{\mathcal W}}
\newcommand{\DD}{{\mathcal D}}
\newcommand{\lamb}{\lambda}

\newcommand{\cH}{H} 
\newcommand{\cR}{{\mathcal R}}
\newcommand{\OP}{{\mathcal L}}
\newcommand{\LL}{{\mathcal P}}
\newcommand{\LO}{{\mathcal K}}
\newcommand{\OO}{{\mathcal M}}
\newcommand{\MM}{{\mathcal B}}
\newcommand{\PP}{{\mathcal Q}}
\newcommand{\TT}{{\mathcal T}}
\newcommand{\cS}{{\mathcal S}}
\newcommand{\cF}{{\mathcal F}}
\newcommand{\LQ}{L}
\newcommand{\cf}{\psi}

\newcommand{\hpg}[5]{{}_{#1}\mbox{\rm F}_{\!#2}\!
  \left(\left.{#3 \atop #4}\right| #5 \right) }
 \newcommand{\hpgg}[4]{{}_{#1}\mbox{\rm F}_{\!#2}{\,#4\, \choose #3} }
\newcommand{\hpgo}[2]{{}_{#1}\mbox{\rm F}_{\!#2}}
\newcommand{\exxp}[1]{\mbox{\bf e}^{#1}}
\newcommand{\frc}[2]{\mbox{\small $\displaystyle\frac{#1}{#2}$}}
\newcommand{\dif}[1]{\mbox{\small $\displaystyle\frac{\partial}{\partial #1}$}}
\newcommand{\diff}[2]{\mbox{\small $\displaystyle\frac{\partial^#2}{\partial #1^#2}$}}

\title{Differential relations for the largest root distribution of complex non-central Wishart matrices}

\author{
        Raimundas Vidunas\footnote{
        Osaka University, Osaka, Japan.
        E-mail: {\sf rvidunas@gmail.com}.}
\hspace*{1cm}
        Akimichi Takemura\footnote{
        Shiga University. E-mail: {\sf akimichi.takemura@gmail.com }.}
       }

\begin{document}

\date{}
\maketitle    

\begin{abstract}
A holonomic system for the probability density function
of the largest eigenvalue of a non-central complex Wishart distribution 
with identity covariance matrix  is derived.
Furthermore a new determinantal formula for the probability density function is derived 
(for $m=2,3$) or conjectured.
\end{abstract}

\section{Introduction}

The Wishart distribution is an important higher dimensional generalization 
of the  $\chi^2$-distribution. 
In many applications the  distribution of {\em roots} (i.e., eigenvalues) of Wishart matrices are needed
(see references in Hashiguchi et al.\ \cite{HNTT2013}).
In this paper we consider complex non-central Wishart matrices, which are important for
applications to performance evaluation of wireless communication systems (Siriteanu et al.\ \cite{siriteanu_twc_SC_14}, \cite{siriteanu_twc_14}).
The purpose of this paper is to give differential relations for the largest root distribution of complex non-central Wishart matrices based on the result of Kang and Alouini \cite{KangAlouini2003}.

Suppose we take $n$ random vectors $x_i\in \CC^m$, $i=1,\dots,n$, independently drawn from 
an $m$-variate complex 
Gaussian distribution ${\mathcal C}\NN(v_i,\Sigma)$, 
with the mean vector $v_i$ and the covariance matrix $\Sigma$. 
We put those vectors into $n\times m$ matrices $X$ and $V$.
The distribution of the random (symmetric, positive  definite) 
$m\times m$ scatter matrices  $S=X^*X$ 
defines the  complex 
{\em Wishart distribution}  $\WW_m(\Sigma,V^*V\Sigma^{-1},n)$ 
with degrees of freedom $n$, covariance matrix $\Sigma$ 
and the non-centrality parameter matrix $V^*V\Sigma^{-1}$.
We are interested in the distribution of largest root 
of $S$.

In the special case $m=1$, we have the distribution
of the value $|x_1|^2+\ldots+|x_n|^2$. In the $\RR$-valued case,
this is the $\chi^2$-distribution.
The central $\chi^2$-distribution ($V=0$)  is a special case 
of the gamma distribution. 

The distribution  of the largest root 
of the  $\RR$-valued central Wishart distribution is known, Muirhead \cite{Muirhead1970}.
The probability distribution function for the largest root
is expressed in terms of a matrix hypergeometric function: 
\begin{equation}
\frac{\Gamma_m\!\left(\frac{m+1}2\right)(\det \Sigma)^{-\frac{n}{2}}}
{\Gamma_m\!\left(\frac{n+m+1}2\right)}
\exp\!\left(-\frc{x}2 \, \mbox{tr}\,\Sigma^{-1}\right) \! \left(\frc{x}2\right)^{\!\frac{nm}2} 
\hpg11{\frac{m+1}2}{\frac{n+m+1}2}{\,\frc{x}2\,\Sigma^{-1}}.
\end{equation}
Here $\Gamma_m(z) = \pi^{\frac{1}{4}m(m-1)}
\prod_{i=1}^{m}\Gamma \left(z-\frac{i-1}{2}\right)
$ 
is called the multivariate Gamma function 
and the $\hpgo11(M)$ function is defined in terms
of symmetric functions ({\em zonal polynomials}) of the eigenvalues of $M$;
Constantine \cite{Constantine1963}, James \cite{James1964}. 

A holonomic system for $\hpg11{a}{c}{M}$ in terms of the eigenvalues $\lamb_i$ of $M$
was derived by Muirhead \cite{Muirhead1970}. For $i\in\{1,2,\ldots,m\}$ we have
\begin{align}
\lamb_i\diff{\lamb_i}{2}+\left(c-\lamb_i \right) \dif{\lamb_i}
+\frc12\sum_{j=0,j\neq i}^m \frc{\lamb_j}{\lamb_i-\lamb_j} \left(\dif{\lamb_i}-\dif{\lamb_j}\right)-a.
\end{align}
Efficiency of the holonomic gradient method was demonstrated 
by Hashiguchi et al.\  \cite{HNTT2013}.

In Section \ref{sec:main}
we derive differential relations for the 
density function of the largest root of 
complex non-central Wishart matrices with the identity covariance matrix $\Sigma=\text{Id}$.
Additionally we assume the Gaussian distribution to be {\em circularly symmetric};
see \cite{paulraj-etal-2004}, \cite[Complex normal distribution]{WikiWishart}).
A conjectural formula is given in Section \ref{sec:conj}.
Later sections are devoted to proofs and additional results for $m\le 3$.

\section{Setting and the contributions}

The cumulative distribution function $\cF_{n,m}(x)$ for the largest root $x$ 
in the circularly symmetric case
was derived by  Kang and Alouini \cite{KangAlouini2003}. Let us recall the 
hypergeometric function
\begin{align} \label{eq:hpg01def}
\hpgg01{n}{z}:= & \; \hpg01{-}{n\,}{z} = 
\sum_{k=0}^{\infty} 
\frac{x^k}{(n)_k\,k!},
\end{align}  
where $(n)_k=n(n+1)\ldots(n+k-1)$ is the Pochhammer symbol \cite{WikiWishart}.
This function is related to the non-central $\chi^2$-distribution 
and the modified Bessel function \cite[\S 9]{NicoTemme}
\begin{equation}
I_{n}(z) = \frac{(z/2)^n}{n!}\,\hpgg01{n+1}{z^2/4}.
\end{equation}
We introduce the integral
\begin{align} \label{eq:hkndef}
\cH^k_n(x,y)= & \, \int_0^x  t^k\, \exxp{-t} \, \hpgg01{n}{ty} dt
\end{align} 
related to the Marcum $Q$-function  \cite[\S 9]{NicoTemme}
\begin{equation}
Q_{n}(x,y) = \frac{\exxp{-x}}{(n-1)!}\,\int_y^{\infty} t^{n-1}\exxp{-t}\hpgg01{n}{xt} dt.
\end{equation}

\subsection{The distribution functions}

Let $\lamb_1,\ldots,\lamb_m$ be the eigenvalues of $V^*V$.
The Kang--Alouini distribution function 
for the largest root of $\WW_m(\mbox{Id},V^*V,n)$ is
\begin{align} \label{eq:cdfdef}
\hspace{-5pt}
\cF_{n,m}(x,\lamb_1,\ldots,\lamb_m) = \!
\frac{\exxp{-\lamb_1-\ldots-\lamb_m}} 
{\mbox{\small$\displaystyle \! \{(n-\!m)!\}^m\,\prod_{i=1}^m\,\prod_{j>i}^{m}  (\lamb_i\!-\!\lamb_j)$}}
\det \! \Big( \cH_{n-m+1}^{n-j}(x,\lambda_i) \; \big\rangle_{i=1}^m \Big).
\end{align}
Here $\rangle_{i=1}^m$ indicates the $i$-th row of an $m\times m$ matrix,
with the column index implicitly taken to be $j$.

The probability density function 
is 
\begin{align}  \label{eq:kanaf}
\cf_{n,m}(x,\lamb_1,\ldots,\lamb_m)  = & \, 
\frac{\partial}{\partial x} \cF_{n,m}(x,\lamb_1,\ldots,\lamb_m) \\
\label{eq:kanafz}
 = &\, \frac{\exxp{-\lamb_1-\ldots-\lamb_m}}
{\mbox{\small$\displaystyle \! \{(n-\!m)!\}^m\,\prod_{i=1}^m\,\prod_{j>i}^{m}  (\lamb_i\!-\!\lamb_j)$}}\, 
\cR_{n,m}(x,\lamb_1,\ldots,\lamb_m), 
\end{align}
where
\begin{align}  \label{eq:defr}
\hspace{-6pt} \cR_{n,m}(x,\lamb_1,\ldots,\lamb_m) = &\, \frac{\partial}{\partial x}
\det \Big( \cH_{n-m+1}^{n-j}(x,\lambda_i) \; \big\rangle_{i=1}^m \Big) \\
  \label{eq:defr2}
= & \, \exxp{-x}  \sum_{k=1}^m \hpgg01{n\!-\!m\!+\!1}{x\lamb_k} 
\det \! \left(   \begin{array}{@{}c@{\;}c@{}}
\cH^{n-j}_{n-m+1}(x,\lamb_i) & \rangle^{i\leq m}_{i\neq k} \\
x^{n-j}  & \rangle_{i=k} 
\end{array} \right). 
\end{align}
For example, an expanded expression for $m=3$ is
\begin{align} 
& \hspace{-8pt} \cR_{n,3}(x,\lamb_1,\lamb_2,\lamb_3)= \nonumber \\
& \; \exxp{-x} \, \hpgg01{n-2}{x\lamb_1}  \det \! 
\left(   \begin{array}{ccc}
\, x^{n-1} & \, x^{n-2} & \, x^{n-3}  \\
\! \cH^{n-1}_{n-2}(x,\lamb_2) &  \cH^{n-2}_{n-2}(x,\lamb_2) & \cH^{n-3}_{n-2}(x,\lamb_2) \! \\[2pt]
\! \cH^{n-1}_{n-2}(x,\lamb_3) &  \cH^{n-2}_{n-2}(x,\lamb_3) & \cH^{n-3}_{n-2}(x,\lamb_3) \!
\end{array} \right)  \\
&+  \exxp{-x} \, \hpgg01{n-2}{x\lamb_2}  \det \! 
\left(  \begin{array}{ccc}
\! \cH^{n-1}_{n-2}(x,\lamb_1) & \cH^{n-2}_{n-2}(x,\lamb_1) & \cH^{n-3}_{n-2}(x,\lamb_1) \! \\[1pt]
\, x^{n-1} & \, x^{n-2} & \, x^{n-3}  \\[1pt]
\! \cH^{n-1}_{n-2}(x,\lamb_3) &  \cH^{n-2}_{n-2}(x,\lamb_3) & \cH^{n-3}_{n-2}(x,\lamb_3) \! 
\end{array} \right)  \nonumber \\
&+ x^{n-3}\, \exxp{-x} \, \hpgg01{n-2}{x\lamb_3}  \det \! 
\left(  \begin{array}{ccc}
\! \cH^{n-1}_{n-2}(x,\lamb_1) & \cH^{n-2}_{n-2}(x,\lamb_1) & \cH^{n-3}_{n-2}(x,\lamb_1) \! \\[2pt]
\! \cH^{n-1}_{n-2}(x,\lamb_2) & \cH^{n-2}_{n-2}(x,\lamb_2) & \cH^{n-3}_{n-2}(x,\lamb_2) \! \\
\, x^2 & \, x & \, 1
\end{array} \right)\!.  \nonumber
\end{align}

\subsection{Main results}
\label{sec:main}

The main result of this paper is a holonomic system of differential equations
for $\cR_{n,m}(x,\lamb_1,\ldots,\lamb_m)$, for any dimension $m$. 
It is formulated in the following two theorems. The first one introduces 
a holonomic system with the differentiations $\partial/\partial \lamb_i$ only.
Theorem \ref{th:elimx2} allows to introduce or eliminate 
$\partial/\partial x$.

Recall \cite{PVdef} 
that a {\em least common left multiple} (LCLM) of several differential operators 
$\OP_1,\ldots,\OP_k$ in the Weyl algebra $\CC(x,y)\langle \partial/\partial y\rangle$ is
a differential operator $\OP^*$ of minimal order such that $\OP^*$ is a left multiple of any $\OP_j$,
$j\in \{1,\ldots,k\}$.
An alternative defining property is that $\OP^*Y=0$ is a differential equation of minimal order
such that all Picard-Vessiot \cite{PVdef} solutions of $\OP_jY=0$ are solutions of $\OP^*Y=0$.
\begin{theorem} \label{th:difops}
Let us define the differential operators
\begin{align}
\LL_{M}[y]= &\, y\,\diff{y}{2}+(M+1)\,\dif{y}-x, \\
\PP_{N,M}[y] = &\,   y\,\diff{y}3+(M-y+2)\,\diff{y}2
-(x+N+1)\,\dif{y}+x.
\end{align}
Let us denote $\TT_1[y]=\LL_{n-m}[y]$ and
\begin{align}
\TT_j[y]=\PP_{n-m+j,n-m}[y] \qquad \mbox{for } 2\le j\le m. 
\end{align} 
The following operators annihilate $\cR_{n,m}(x,\lamb_1,\ldots,\lamb_m)$:
\begin{enumerate}
\item The products $\TT_k[\lambda_1]\cdots\TT_k[\lambda_m]$, for $k=1,2,\ldots,m$. 
\item The least common left multiples
\mbox{\rm LCLM}$(\TT_1[\lamb_k],\ldots,\TT_m[\lamb_k])$  
with $k=1,2,\ldots,m$.
\end{enumerate} 
\end{theorem}
\begin{theorem} \label{th:elimx2}
This second order operator annihilates $\cR_{n,m}(x,\lamb_1,\ldots,\lamb_m)$:
\begin{equation} \label{eq:elimx2}
x\dif{x}+\sum_{k=1}^m 
\Big( \lamb_k \diff{\lamb_k}{2} 
+\big(n-m+1-\lamb_k\big)\dif{\lamb_k} -n\Big) 
+\frac{m(m-1)}2+1.
\end{equation}
\end{theorem}
The theorems are proved in \S \ref{sec:difops} and \S \ref{sec:elimx2}. 
To get differential equations 
for the density function $\cf_{n,m}(x,\lamb_1,\ldots,\lamb_m)$,
the presented operators 
must be modified by the {\em gauge translations}
\begin{align}
\dif{\lamb_i}\mapsto \dif{\lamb_i}+1+\sum_{j\neq i} \frac{1}{\lamb_i-\lamb_j}.
\end{align}
This is a standard technique to account for the front factor in (\ref{eq:kanafz}).

By its determinantal form (\ref{eq:defr}), 
the target function $\cR_{n,m}(x,\lamb_1,\ldots,\lamb_m)$  is a non-logarithmic
and anti-symmetric function. In particular, it is multiplied by the sign $(-1)^\sigma$ 
under a permutation $\sigma$ of the variables $\lamb_1,\ldots,\lamb_m$
\begin{theorem} \label{th:hrank}
\begin{enumerate}
\item The differential operators of Theorem $\ref{th:difops}$ 
annihilating \hfill \\ $\cR_{n,m}(x,\lamb_1,\ldots,\lamb_m)$ 
 generate a holonomic system of rank $2m!\cdot 3^{m-1}$.
\item Let $\cS$ denote the subspace of anti-symmetric solutions 
in a full solution space (of dimension $2m!\cdot 3^{m-1}$). Then $\dim\cS=2\cdot 3^{m-1}$.
\item The subspace of non-logarithmic anti-symmetric solutions has the dimension $2^{m-1}$.
\item There exists a holonomic system of rank $\le 3^m-1$ defined over $\QQ(x,\lamb_1,\ldots,\lamb_m)$ 
and annihilating $\cR_{n,m}(x,\lamb_1,\ldots,\lamb_m)$.
\end{enumerate}
\end{theorem}
This theorem is proved in \S \ref{sec:holos}.
Our computations for $m=2$, $m=3$ indicate that the lower rank system
has markedly more complicated equations and singularities.
These computations 
are presented in \S \ref{sec:m2}.

\subsection{A conjectural formula}
\label{sec:conj}

We were led to Theorem \ref{th:difops} after computing holonomic systems
for $\cR_{n,2}(x,\lamb_1,\lamb_2)$ of rank 12 and 8, elimination of 
$\partial/\partial x$, $\partial/\partial\lamb_2$ and observing a differential
operator in $\partial/\partial\lamb_1$ of order 5 with a simple LCLM factorization. 
Computations for $\cR_{n,3}(x,\lamb_1,\lamb_2,\lamb_3)$
led to holonomic systems of rank 108 and 26 cumbersomely,
but probing for a differential operator in only $\partial/\partial\lamb_1$
quickly gave one of relatively low order 8 and a remarkable LCLM factorization
into operators of order 2 or 3.
Theorem \ref{th:difops} establishes continuation of this pattern.

The solution space of the holonomic systems in Theorem \ref{th:difops} 
is highly factorizable by specificity of the presented generators. 
Particular solutions are
\begin{equation} \label{eq:detconj}
\det \Big( \,Y_j(x,\lambda_i) \; \big\rangle_{i=1}^m \Big),
\end{equation}
where $Y_j(x,y)$ is a solution $\TT_j[y]\,Y_j=0$. 
The LCLM operator in \refpart{ii} annihilates the $k$th row of this matrix, 
while the product in \refpart{i} annihilates the $k$th column.
Based on obtained new expressions for $\cR_{n,2}(x,\lamb_1,\lamb_2)$,
$\cR_{n,3}(x,\lamb_1,\lamb_2,\lamb_3)$, solutions of $\PP_{N,M}[y]$ and their recurrences,
we conjecture  that $\cR_{n,m}(x,\lamb_1,\ldots,\lamb_m)$ 
has a determinantal expression (\ref{eq:detconj}).
Here is a formulation in the transposed form.
\begin{conjecture} \label{th:conjecture}
Let us define 
\begin{align} \label{eq:osol2}
G_{n,2}(x,y) = & \; n\; \hpgg01{n}{xy}+y \;\hpgg01{n+1}{xy} \nonumber \\
& +(x-y-n+1) \, \exxp{y} \int_y^{\infty} \! \exxp{-t} \, \hpgg01{n+1}{xt} dt.  \qquad
\end{align}
For $m\ge 2$, we recursively define
\begin{equation} \label{eq:grec}
G_{n,m+1}(x,y) = \left( -y\,\diff{y}2-(n-m+1)\,\dif{y}+x+m \right) G_{n,m}(x,y).
\end{equation}
We conjecture that
\begin{equation} \label{eq:conj}
\cR_{n,m}(x,\lamb_1,\ldots,\lamb_m)
= C(x)\,  \det \! \left(   \begin{array}{@{}c@{\;}c@{}}
 \hpgg01{n-m+1}{x\lamb_i} & \rangle_{j=1} \\[2pt]
G_{n-m+j,j}(x,\lamb_i) & \rangle_{j=2}^m
\end{array} \right)
\end{equation}
with
\begin{equation} \label{eq:ff}
C(x) = \frac{(n-m+1)\,x^{mn-{m\choose 2}-1}\,\exxp{-mx}}{\prod_{k=1}^m (n-k+1)^k}.
\end{equation}
\end{conjecture}
Note that $\hpgg01{n-m+1}{x\,y}$ is a solution of $\LL_{n-m}[y]\,Y=0$.
As we show in \S  \ref{sec:solrec},
the function $G_{n,m}(x,y)$ is a solution of $\PP_{n-m+j,n-m}[y]=0$
for any integers $n\ge m> 0$. 
Recurrence (\ref{eq:grec}) stems from \S \ref{sec:solrec} as well.

Notably, the integral in (\ref{eq:osol2}) is complementary to $\cH_{n+1}^0(y,x)$.

The conjecture has been fully checked for $m=2$ and $m=3$, 
as described in \S \ref{sec:solrecp}  and \S \ref{sec:m3}. 
Also, the front factor (\ref{eq:ff}) has been confirmed for $m=4$.
The conjecture happens to be true for $m=1$ as well.
In  \S \ref{sec:solrecp} we specifically prove
\begin{equation}  \label{th:om2}
\cf_{n,2}(x,\lamb_1,\lamb_2)= \frac{x^{2n-2}\,\exxp{-\lamb_1-\lamb_2-2x}}{n!\,(n-2)!\,(\lamb_1-\lamb_2)} 
\, \det \! \left( \begin{array}{cc}
\hpgg01{n-1}{x\lamb_1} & \hpgg01{n-1}{x\lamb_2}  \\[2pt]
\! G_{n,2}(x,\lamb_1) & G_{n,2}(x,\lamb_2) \!
\end{array} \right)\!.
\end{equation}
If $\lamb_1=\lamb_2$, application of l'Hospital's rule leads to 
differentiating a matrix column. 
For comparison, expression (\ref{eq:kanaf}) is
\begin{align} \label{eq:start}
\cf_{n,2}(x,\lamb_1,\lamb_2) = \frac{\exxp{-\lamb_1-\lamb_2}}{\{(n-2)!\}^2} \; &
\frac{ \dif{x} 
\det \left( \begin{array}{cc}
\cH^{n-1}_{n-1}(x,\lamb_1) & \cH^{n-1}_{n-1}(x,\lamb_2) \\[3pt]
\cH^{n-2}_{n-1}(x,\lamb_1) & \cH^{n-2}_{n-1}(x,\lamb_2)
\end{array} \right) }{\lamb_1-\lamb_2}.
\end{align}
Not only the differentiation $\partial/\partial x$ is avoided, 
but the integral in (\ref{eq:osol2}) is numerically preferable to 
the $\cH^{n-1}_{n-1}$, $\cH^{n-2}_{n-1}$ functions.

Significance of the conjectured formula is that it would utilize the factorization structure
of the holonomic system in Theorem \ref{th:difops}. 
Applying the holonomic gradient method to the entries of the conjectured matrix 
would be more efficient than employing the same method
for the large multi-variate holonomic system.

\subsection{Auxiliary integrals}

To get the holonomic system, 
we use recurrences for  $\cH^k_n(x,y)$ in (\ref{eq:hkndef})  and the generalization
\begin{equation} \label{eq:hkln}
\cH^{k,\ell}_{n}(x,y)=\int_0^x \! 
\exxp{-t} \, t^{k} (x-t)^{\ell}\,\hpgg01{n}{t\,y}dt.
\end{equation}
Surely, $\cH^{k,0}_{n}(x,y)=\cH^{k}_{n}(x,y)$. 
These differentiations are straightforward:
\begin{align} \label{eq:difx}
\frac{\partial}{\partial x} \cH^{k}_n(x,y) = & \, x^k\exxp{-x} \, \hpgg01{n}{xy},\\
\label{eq:dify}
\frac{\partial}{\partial y} \cH^{k,\ell}_n(x,y) = & \, \frac{1}{n}\, \cH^{k+1,\ell}_{n+1}(x,y).
\end{align}
\begin{lemma} \label{lm:hkn}
If $k>0$, then
\begin{align}  \label{eq:rec3}
\cH^k_{n-1}(x,y) = & \, \cH^k_n(x,y) +\frac{y}{n(n-1)}\,\cH^{k+1}_{n+1}(x,y), \\
  \label{eq:recip}
k \, \cH^{k-1}_n(x,y) = & \, \cH^k_n(x,y) -\frac{y}{n}\,\cH^{k}_{n+1}(x,y)
 +x^k \exxp{-x} \, \hpgg01{n}{xy}, \\ 
 \label{eq:rechd}
k\,\cH^{k-1}_{n-1}(x,y)= & \, \frac{n\!-\!y\!-\!1}{n-1}\, \cH^k_n(x,y)
+\frac{y}{n(n\!-\!1)}\,\cH^{k+1}_{n+1}(x,y)  
+x^k \exxp{-x} \hpgg01{n\!-\!1}{xy}.
\end{align}
\end{lemma}
\begin{proof}
The first formula follows from the recurrence 
\begin{equation}
\hpgg01{n-1}{z}= \hpgg01{n}{z} +
\frac{z}{n(n-1)} \, \hpgg01{n+1}{z}
\end{equation}
that is equivalent to the hypergeometric equation (with $a=n$)
\begin{equation} \label{eq:hpgde}
z\,Y''(z)+a\,Y'(z)-Y(z)=0
\end{equation}
for $\hpgo01(z)$. The second formula follows after integration by parts  
\begin{equation}
\cH_n^{k}(x,y)=-\!\int_0^x  t^k  \hpgg01{n}{ty} d\exxp{-t}.
\end{equation}
The last formula follows after substituting $n\mapsto n-1$ in (\ref{eq:recip})
and eliminating $\cH^k_{n-1}(x,y)$ using (\ref{eq:rec3}). 
\end{proof}
Formula (\ref{eq:rechd}) is equivalent to the differential equation
\begin{equation} \label{eq:diffhjk}
\Big( y \frac{\partial^2}{\partial y^2}+ (n-y) \frac{\partial}{\partial y}-k-1 \Big) \,
\cH^{k}_{n}(x,y) = -x^{k+1} \exxp{-x} \hpgg01{n}{xy}.
\end{equation}
We can obtain recurrences that shift only the indices $k$ or $n$,
presented in the following lemma.
Remarkably, both formulas lose an $\cH$-term when $k=n-1$.
The simplified formulas are readily applicable to the $j=m$, $j=m-1$ columns in (\ref{eq:cdfdef}).
\begin{lemma} \label{lm:hkn2}
If $k>0$, then
\begin{align} 
\hspace{-6pt} n(n\!-\!1) \cH^{k}_{n-1}(x,y) = & \, n\,(y+n-1)\cH^k_n(x,y)+
 y\,(k-n+1)\cH^k_{n+1}(x,y) \nonumber \\
 & -yx^{k+1} \exxp{-x} \, \hpgg01{n+1}{xy},\\
\cH^{k+1}_{n}(x,y)= & \, 
(y-n+2k+2)\cH^k_n(x,y)+k\,(n-k-1)\cH^{k-1}_{n}(x,y) \nonumber \\
& \! -(n\!-\!1)x^k \exxp{-x}\hpgg01{n\!-\!1}{xy}+(k\!-\!x)x^k \exxp{-x} \hpgg01{n}{xy}.
 \end{align}
 \end{lemma}
\begin{proof}
The first formula is obtained by eliminating $\cH^{k-1}_{n}(x,y)$, $\cH^{k-1}_{n-1}(x,y)$
from these 3 equations: (\ref{eq:recip}), the shift $k\mapsto k-1$ of (\ref{eq:rec3}),
and the shift $n\mapsto n-1$ of (\ref{eq:recip}). 
For the second formula, we eliminate $\cH^{k}_{n+1}(x,y)$, $\cH^{k-1}_{n-1}(x,y)$,
$\cH^{k+1}_{n+1}(x,y)$ from these 4 equations: (\ref{eq:recip}), (\ref{eq:rechd}), 
the shift $k\mapsto k-1$ of (\ref{eq:rec3}), and the shift $k\mapsto k+1$ of (\ref{eq:recip}). 
\end{proof}
\begin{lemma} \label{th:simprec}
The following recurrences with two $\cH$-terms hold, for $n>0$:
\begin{align}  \label{eq:simrec1}
(n\!-\!1)\cH^{n-1}_{n-1}(x,y)= & \, (y+n-1)\cH^{n-1}_n(x,y)
-\frac{yx^{n}}{n} \exxp{-x} \, \hpgg01{n+1}{xy}, \\   
\label{eq:simrec2}
\cH^{n}_{n}(x,y)= & \, (y+n)\cH^{n-1}_n(x,y)
-\frac{yx^{n}}{n} \exxp{-x} \, \hpgg01{n\!+\!1}{xy}  \nonumber \\ &  
-x^{n} \exxp{-x} \,\hpgg01{n}{xy}, \\
\label{eq:hrecg}
\cH^{n}_{n+1}(x,y)= &\, n \, \cH^{n-1}_n(x,y)-x^n \exxp{-x} \hpgg01{n+1}{xy}.
\end{align}
\end{lemma}
\begin{proof} 
The first two formulas constitute the special case $k=n-1$ of the previous lemma.
The third formula is obtained by eliminating $\cH^{n}_{n}(x,y)$ from (\ref{eq:simrec2})
and the shift $n\mapsto n+1$ of (\ref{eq:simrec1}). 
\end{proof}

Formula (\ref{eq:hrecg}) is comparable to the recurrence for the incomplete gamma function 
$\gamma(a,x)=\int_0^x  t^{a-1} \exxp{-t}\,dt$:
\begin{equation} \label{incgamma}
\gamma(a+1,z)= a \, \gamma(a,z)-x^a \exxp{-x}.
\end{equation}
The presented recurrences 
can be used to express all matrix entries in (\ref{eq:cdfdef})
in terms of $\cH^{n-m}_{n-m+1}(x,\lamb_i)$ and two $\hpgo01$ functions.
\begin{proposition} \label{lm:basis3}
Any function $\cH^{k}_{n}(x,y)$ with integer $n>1$ and $k\ge n-1$
can be expressed as a $\QQ(x,y)$-linear combination of
\[
\cH_N^{N-1}(x,y), \quad x^N\exxp{-x}\hpgg01{N}{xy}, \quad x^N\exxp{-x}\hpgg01{N+1}{xy}
\]
for any $N>1$.
The same statement applies to the derivatives of $\cH^{k}_{n}(x,y)$ of any order (with respect to $x,y$).
\end{proposition}
\begin{proof} Lemma \ref{th:simprec} proves the first claim for $k=n-1$ and $k=n$.
Lemma \ref{lm:hkn2} extends the statement to larger $k$. 
Differentiation rules (\ref{eq:difx})--(\ref{eq:dify}) imply the second claim.
\end{proof}

Recurrence relations  $\cH^{k,\ell}_{n}(x,y)$ are obtained by
a straightforward extension of the results for $\cH^{k}_{n}(x,y)$.
\begin{lemma} \label{th:hklnrecs}
For $k>0$, $\ell>0$ we have
\begin{align} \label{eq:hklnx}
x\,\cH^{k,\ell}_{n}(x,y)= &\, \cH^{k+1,\ell}_{n}(x,y)+\cH^{k,\ell+1}_{n}(x,y),\\
 \label{eq:hklnr}
\cH^{k,\ell}_{n-1}(x,y)= &\, \cH^{k,\ell}_{n}(x,y)+\frac{y}{n(n-1)}\,\cH^{k+1,\ell}_{n+1}(x,y), \\
 \label{eq:hklni}
\cH^{k,\ell}_{n}(x,y)= &\, k\,\cH^{k-1,\ell}_{n}(x,y)-\ell\,\cH^{k,\ell-1}_{n}(x,y)
+\frac{y}{n}\,\cH^{k,\ell}_{n+1}(x,y).
\end{align}
\end{lemma}
\begin{proof}
The first recurrence is obtained by splitting 
\[
(x-t)^\ell=x(x-t)^{\ell-1}-t(x-t)^{\ell-1}
\]
in the defining integral (\ref{eq:hkln}).
The other two equations follow similarly as (\ref{eq:rec3})--(\ref{eq:recip}),
from the three-term recurrence for the $\hpgo01$ function and,
respectively, by integration by parts.
\end{proof}
\begin{lemma}
For $k>0$, $\ell>0$ we have
\begin{align}
\label{eq:hklrecu}
\! (n\!-\!1)\cH^{k-1,\ell}_{n-1}(x,\!y)
= &\, \cH^{k,\ell}_{n}(x,\!y)+\ell\,\cH^{k,\ell-1}_{n}(x,\!y)+(n-k-1)\cH^{k-1,\ell}_{n}(x,\!y),\\
\label{eq:hklrecu2}
k\,\cH^{k-1,\ell}_{k}(x,y)
= &\, \cH^{k,\ell}_{k+1}(x,y)+\ell\,\cH^{k,\ell-1}_{k+1}(x,y).
\end{align}
\end{lemma}
\begin{proof}
First we show this intermediate equation:
\begin{align} \label{eq:hklnt} \hspace{-5pt}
\cH^{k,\ell+1}_{n-1}(x,\!y)+\frac{y}{n(n\!-\!1)}\,\cH^{k+2,\ell}_{n+1}(x,\!y) =
\cH^{k,\ell+1}_{n}(x,\! y) +\frac{xy}{n(n\!-\!1)}\,\cH^{k+1,\ell}_{n+1}(x,\!y). 
\end{align}
It is annihilated by the relations of Lemma \ref{th:hklnrecs} as follows.
The two terms with denominators are eliminated by (\ref{eq:hklnr}) and its shift $k\mapsto k+1$.
Then elimination of $\cH^{k,\ell}_{n}(x,y)$, $\cH^{k,\ell}_{n-1}(x,y)$
by (\ref{eq:hklnx}) and its shift $n\mapsto n-1$ leaves no terms.

Now multiply equation (\ref{eq:hklnt}) by $(n-1)$ and apply the shifts $k\mapsto k-1$, $\ell\mapsto \ell-1$.
Then subtract (\ref{eq:hklni}) and eliminate the terms with denominators using
the shifted version $n\mapsto n+1$, $\ell\mapsto \ell-1$ of (\ref{eq:hklnx}).
The result is (\ref{eq:hklrecu}).
The second claimed recurrence is the special case  $n=k+1$ of  the first one.
\end{proof}

\section{Proofs and analysis}
\label{sec:holonomic}

The motivation for this article was potential application of the holonomic gradient method 
\cite{HNTT2013} to computation of the probability density function
$\cf_{n,m}(x,\lamb_1,\ldots,\lamb_m)$.
Our main results are formulated in \S \ref{sec:main}
for the function $\cR_{n,m}(x,\lamb_1,\ldots,\lamb_m)$
as in (\ref{eq:kanafz}). The obtained holonomic systems 
are more complicated than expected, in the simplest cases $m\le 3$ as well. 
As mentioned at the end of \S \ref{sec:conj}, 
application of the holonomic gradient method to the entries 
of the conjectured matrix in (\ref{eq:conj}) should be more effective
than employment of multi-variate holonomic systems.

This section proves the main results presented in \S \ref{sec:main}.
Additionally, \S \ref{sec:holos} discusses the obtained holonomic systems,
and \S \ref{sec:solrec} presents explicit solutions and recurrences relevant to Conjecture \ref{th:conjecture}.


\subsection{Proof of Theorem \ref{th:difops}}
\label{sec:difops}
 
The product $\TT_1[\lamb_1]\ldots\TT_1[\lamb_m]$
annihilates $\cR_{n,m}(x,\lamb_1,\ldots,\lamb_m)$ 
because the operators $\TT_1[\lamb_i]=\LL_{n,m}[\lamb_i]$ annihilate
the front factor $\hpgg01{n-m+1}{x\lamb_i}$ of each term in (\ref{eq:defr}).
Remarkably, the other operators annihilate each term in (\ref{eq:defr})
as well. 

Consider now the action of $\TT_m[y]=\PP_{n,m}[y]$. We have
\begin{align} \label{eq:tmonr}
\TT_m[\lamb_1]\ldots& \TT_m[\lamb_m]\,\cR_{n,m}(x,\lamb_1,\ldots,\lamb_m)= \\ 
& \exxp{-x} \, \sum_{k=1}^m \TT_m[\lamb_k] \hpgg01{n\!-\!m\!+\!1}{x\lamb_k} 
\det \! \left(   \begin{array}{@{}c@{\;}c@{}}
\TT_m[\lamb_i]\cH^{n-j}_{n-m+1}(x,\lamb_i) & \rangle^{i\leq m}_{i\neq k} \\
x^{n-j}  & \rangle_{i=k} \end{array} \right).   \nonumber
\end{align}
We claim that all $m$ determinants are zero,
because the matrices have a specific kernel vector
\begin{equation}
\left( \, {\textstyle {m-1\choose j-1} } (-x)^j \; \big\rangle_{j=1}^{m} \,\right)^{\!T}.
\end{equation}
The scalar product of this vector with the $i=k$ rows 
\begin{equation} \label{eq:xkrows}
\left( x^{n-1}, x^{n-2}, \ldots, x^{n-m} \right)^T
\end{equation}
equals $0$ straightforwardly, since $\sum_{j=1}^m (-1)^j {m-1\choose j-1}=0$ as well known.
We want to show
\begin{equation} \label{eq:vectrel}
\sum_{j=1}^m {m-1\choose j-1} (-x)^{j} \;
  \TT_m[y]\cH^{n-j}_{n-m+1}(x,y)=0.
\end{equation} 
By applying the differentiation 
\begin{equation}
 \TT_m[y]= \Big(\dif{y}-1\Big) \LL_{n,m}[y]-m\,\dif{y}
\end{equation}
and recurrence 
of Lemma \ref{lm:hkn},
\begin{align*}
\TT_m[y] \cH^{n-j}_{n-m+1}(x,y) = 
& \,  x\,\cH^{n-j}_{n-m+1}(x,y)-\frac{x+j-1}{n-m+1} \,\cH^{n-j+1}_{n-m+2}(x,y) 
\nonumber \\
&  - x^{n-j+2} \exxp{-x} \hpgg01{n-m+2}{x\,y}.
\end{align*}
We ignore the last term, for linear dependency with the $i=k$ row.
By permuting the summation and integration, our target is 
\begin{equation}
\int_0^x \! \exxp{-t} t^{n-m} (x-t)^{m-1} P(x,t) dt=0
\end{equation}
with 
\begin{equation}
P(x,t)=(x-t) \, \hpgg01{n-m+1}{y\,t} 
-\frac{t\,(x+m-t)}{n-m+1} \, \hpgg01{n-m+2}{y\,t}.
\end{equation}
This integral is equivalent to the recurrence relation 
\begin{equation} \label{eq:superh}
(n-m+1)\,\cH^{n-m,\,m}_{n-m+1}(x,y)
=\cH^{n-m+1,m}_{n-m+2}(x,y)+m\,\cH^{n-m+1,m-1}_{n-m+2}(x,y)
\end{equation}
that is equivalent to (\ref{eq:hklrecu}). 
The claimed relation (\ref{eq:vectrel}) follows.

Other products $\TT_q[\lambda_1]\cdots\TT_q[\lambda_m]$ with $2\le q <m$
annihilate $\cR_{n,m}(x,\lamb_1,\ldots,\lamb_m)$ similarly, with the kernel vectors 
\begin{equation} \label{eq:kervg}
\Big( \underbrace{0,\ldots,0}_{m-q},\,
{\textstyle {q-1\choose j-m+q-1} }(-x)^{j} \; \big\rangle_{j=m-q+1}^{m} \,\Big)^{\!T}
\end{equation}
of the $m$ matrices 
\[
\left(   \begin{array}{@{}c@{\;}c@{}}
\TT_q[\lamb_i]\cH^{n-j}_{n-m+1}(x,\lamb_i) & \rangle^{i\leq m}_{i\neq k} \\
x^{n-j}  & \rangle_{i=k} \end{array} \right)
\]
in an expression like in (\ref{eq:tmonr}).

For $\ell=1,\ldots,m$, the LCLM$(\TT_1[\lamb_\ell],\ldots,\TT_m[\lamb_\ell])$ transforms 
$\cR_{n,m}(x,\lamb_1,\ldots,\lamb_m)$ to
\begin{align*}
\exxp{-x} \! \sum_{k=1,\neq\ell}^m \! \hpgg01{n\!-\!m\!+\!1}{x\lamb_k} 
\det \! \left(   \begin{array}{@{}c@{\;}c@{}}
\cH^{n-j}_{n-m+1}(x,\lamb_i) & \rangle^{i\leq m}_{i\neq k,\ell} \\
\mbox{LCLM}\,\cH^{n-j}_{n-m+1}(x,\lamb_i) & \rangle_{i=\ell} \\
x^{n-j}  & \rangle_{i=k} \end{array} \right)\!.
\end{align*}
The $i=\ell$ row is proportional to the $i=k$ row vector (\ref{eq:xkrows}), because:
\begin{itemize}
\item For $q=2,\ldots,m$, the operator $\TT_q[\lambda_\ell]$
makes the $i=\ell$ row ``orthogonal" to (\ref{eq:kervg}).
\item The LCLM is a left factor of each $\TT_q[\lambda_\ell]$,
thus preserves
the ``orthogonality" property.
\item The vector (\ref{eq:xkrows}) is the only vector ``orthogonal" to 
the $m-1$ independent vectors.
\end{itemize}
Hence the LCLM operators annihilate all $m$ terms of 
$\cR_{n,m}(x,\lamb_1,\ldots,\lamb_m)$.

\subsection{Holonomic systems}
\label{sec:holos}

Here we prove Theorem \ref{th:hrank}. 
To simplify technical details,
we posit that differential Galois theory \cite{PVdef} extends  
straightforwardly to the considered holonomic systems.

Let $\OO$ denote the system of differential operators in Theorem \ref{th:difops}
annihilating $\cR_{n,m}(x,\lamb_1,\ldots,\lamb_m)$.
It is a holonomic system, because the LCLM operators bound the order 
in each $\partial/\partial \lamb_i$. Since the rank of an LCLM operator equals $3m-1$,
a straightforward upper bound for the holonomic rank is $(3m-1)^m$.
After a choice of (Picard-Vessiot) solution basis for each $\TT_k[y]$,
the subsystem of LCLM operators has the following straightforward basis of solutions:
$g_{j_1}(\lamb_{1})\cdots g_{j_m}(\lamb_{m})$, where $g_k(y)$ is a basis solution
of $\TT_k$. Let $\MM$ denote this basis 
of $(3m-1)^m$ functions.

The solution space of $\OO$ will be considered inside the span 
of $\MM$. The following $2m!\cdot 3^{m-1}$ functions in $\MM$ 
will be solutions of $\OO$: $g_1(\lamb_{j_1})\cdots g_m(\lamb_{j_m})$, where
$g_k(y)$ is a basis solution of $\TT_k[y]$,
and $(j_1,\ldots,j_m)$ is a permutation of $(1,\ldots,m)$.
Any other element of $\MM$ is not annihilated by at least one operator
in \refpart{i} of Theorem \ref{th:difops}, and a linear combination of these elements
will not be nullified by the same operator(s). 
The claim \refpart{i} of Theorem \ref{th:hrank} follows.
 
The solution space of $\OO$ splits into a direct sum of $2\cdot 3^{m-1}$ subspaces 
that are invariant under the permutations of $\lamb_1,\ldots,\lamb_m$. 
Each of these subspaces gives one independent anti-symmetric solution,
and the claim \refpart{ii} 
follows.

Each operator $\TT_j[y]$ has logarithmic solutions at $y=0$.
A  broad reason is that appearance of $\hpgg01{n}{z}$ functions 
brings ill-determined $\hpgg01{2-n}{z}$. More precisely, 
logarithmic solutions appear in a limit $a\to n$ of the general solution 
\begin{equation}
C'\,\hpgg01{a}{z}+C''\,z^{1-a}\,\hpgg01{2-a}{z}
\end{equation}
of the hypergeometric equation (\ref{eq:hpgde}) with generic $a\in\CC$.
Analysis of local solutions of $\TT_k[y]$ at the singularities $y=0$, $y=\infty$ shows
that the space of non-logarithmic solutions of $\TT_k[y]$ is one-dimensional for $k=1$
and two-dimensional for $k\ge 2$.  Explicit instances in \S \ref{sec:solrec} demonstrate this.
Similarly as above, the space of non-logarithmic solutions for $\OO$ has 
the dimension $2^{m-1}m!$, and the space of non-logarithmic anti-symmetric solutions
has the dimension $2^{m-1}$.

Existence of a holonomic system of rank $\le 3^{m}-1$ of claim \refpart{iv} follows from 
Proposition \ref{lm:basis3}. It allows to express $\cR_{n,m}(x,\lamb_1,\ldots,\lamb_m)$ 
and its derivatives as $\QQ(x,\lamb_1,\ldots,\lamb_m)$-linear combinations of 
$f_1(\lamb_1)\cdots f_m(\lamb_m)$, where each 
\begin{equation}
f_j(\lamb_j)\in \Big\{ \cH_n^{n-1}(x,\lambda_j), x^n\exxp{-x}\hpgg01{n}{x\lamb_j},
x^n\exxp{-x}\hpgg01{n+1}{x\lamb_j}\Big\}.
\end{equation}
These functions generate the space of dimension $3^m$, but the term 
\[
\cH_n^{n-1}(x,\lambda_1)\cdots \cH_n^{n-1}(x,\lambda_m)
\]
does not appear, because $\cR_{n,m}(x,\lamb_1,\ldots,\lamb_m)$ 
is defined after applying $\partial/\partial x$, 
and further differentiations will not bring this term back.
Examples of these linear expressions are given in \S \ref{sec:rank8}.

\begin{remark} \rm
The determinants in (\ref{eq:detconj}) linearly generate 
the space of anti-symmetric solutions in Theorem \ref{th:hrank} \refpart{ii}. 
They form a Grassmanian-like variety in this space.
Taking scalar multiplication of the rows and the whole matrix
into account, the dimension of this variety equals $(2-1)+(m-1)(3-1)+1=2m$.
Similarly, the subvariety of non-logarithmic anti-symmetric determinantal solutions 
has the dimension $(1-1)+\mbox{$(m-1)\cdot1$}+1=m$. 
For $m=2$, comparison with the dimension count in \refpart{iii} of Theorem \ref{th:hrank}
implies that a determinantal formula like (\ref{th:om2}) is inevitable.
\end{remark}

\begin{remark} \rm
In the proof of Theorem \ref{th:difops} we may start 
with any $m-1$ independent vectors $(v^{(q)}_1,\ldots,v^{(q)}_m)$
``orthogonal" to $(x,x^2,\ldots,x^m)^T$
and take for $\TT_2[y],\ldots,\TT_m[y]$ the operators annihilating
$\sum_{j=1}^m v^{(q)}_j\cH^{n-j}_{n-m+1}(x,y)$ up to a term proportional to (\ref{eq:xkrows}).
The alternative operators 
\[
\TT_k[\lambda_1]\cdots\TT_k[\lambda_m], \qquad
\mbox{LCLM}(\TT_1[\lamb_k],\ldots,\TT_m[\lamb_k])
\] 
would generate a holonomic system annihilating $\cR_{n,m}(x,\lamb_1,\ldots,\lamb_m)$ 
by the same reasons. They would have order 3 as well by Proposition \ref{lm:basis3},
but they would be more complicated, with additional singularities.
For example, taking $m=3$ and the vector $(1,-x,0)^T$ gives the differential operator
\begin{equation} \label{eq:altp3} 
\PP_{n,3}[y]+\frac{x}{xy+(n-2)(x+2)}\left(
-y\frac{\partial^2}{\partial y^2}+y\frac{\partial}{\partial y}+x+2\right)
\end{equation}
instead of $\PP_{n-1,2}[y]$ corresponding to $(0,1,-x)^T$.
As shown in \S \ref{sec:m3}, the LCLM operators
are apparently the same as in Theorem \ref{th:difops},
demonstrating powerfully non-uniqueness of LCLM factorization 
in non-commutative Weyl algebras \cite{WikiWishart}. 
But different products in Theorem \ref{th:difops}\refpart{i} lead
to different holonomic systems, of the same rank though. 
\end{remark}

\subsection{Proof of Theorem \ref{th:elimx2}}
\label{sec:elimx2}

We have to prove that 
\begin{equation} \label{eq:opxx}
x\dif{x}+\sum_{k=1}^m 
\Big( \lamb_k\diff{\lamb_k}{2}+(n-m+1-\lamb_k)\dif{\lamb_k} \Big)
\end{equation}
multiplies $\cR_{n,m}(x,\lamb_1,\ldots,\lamb_m)$ by the constant $mn-{m \choose 2}-1$. 
Note that only $\partial/\partial x$ in (\ref{eq:opxx}) splits the summation factors
in (\ref{eq:defr2}) by Leibniz rule, because the determinants do not depend
on the respective $\lamb_k$ in each summand of (\ref{eq:defr2}).

The action of (\ref{eq:opxx}) on the front factor $\exxp{-x}$ in (\ref{eq:defr2}) is multiplication by $-x$.
This is compensated by the action on the $\hpgo01$ factors in (\ref{eq:defr2}), because:
\begin{itemize}
\item $x\dif{x}-\lamb_k\dif{\lamb_k}$ nullifies $\hpgo01(x\lamb_k)$; 
\item $\lamb_k\diff{\lamb_k}{2}+(n-m+1)\dif{\lamb_k}=\LL_{n-m}[\lamb_k]+x$.
\end{itemize}
The action on the determinants gives
\begin{align} \label{eq:difsum}
& \, \exxp{-x} \sum_{k=1}^m  \hpgg01{n\!-\!m\!+\!1}{x\lamb_k} 
\det \!  \left(   \begin{array}{@{}l@{\;}l@{}}
\cH^{n-j}_{n-m+1}(x,\lamb_i) & \rangle^{i\leq m}_{i\neq k} \\
\,(n\!-\!j)\,x^{n-j}  & \rangle_{i=k} 
\end{array}  \right)  \\
&+\exxp{-x} \! \mathop{\sum\sum}_{k,\ell=1,k\neq \ell}^m  \hpgg01{n\!-\!m\!+\!1}{x\lamb_k} 
\det \!  \left(   \begin{array}{@{}l@{\;}l@{}}
\cH^{n-j}_{n-m+1}(x,\lamb_i) & \rangle^{i\leq m}_{i\neq k,\ell} \\
(n\!-\!j\!+\!1)\,\cH^{n-j}_{n-m+1}(x,\lamb_\ell) & \rangle_{i=\ell} \\
\,x^{n-j}  & \rangle_{i=k} 
\end{array}  \right), \nonumber
\end{align}  
because:
\begin{itemize}
\item Applying $\partial/\partial x$ to the rows $i\neq k$ gives 
linear dependence with the $i=k$ row;
\item Formula (\ref{eq:diffhjk}) implies
\begin{align} \hspace{-8pt}
& \Big(  \lamb_\ell\diff{\lamb_\ell}{2}
+ (n-m+1+\lamb_\ell)\,\dif{\lamb_\ell} \Big) 
H_{n-m+1}^{n-j}(x,\lamb_\ell)  \nonumber \\
& \qquad =  (n-j+1) \,H_{n-m+1}^{n-j}(x,\lamb_\ell)  \, 
 - x^{n-j+1}  \exxp{-x} \, \hpgg01{n-m+1}{x\,\lamb_\ell};
\end{align}
\item The new terms with $\hpgo01(x\lamb_\ell)$ can be ignored by linear 
combination with the $i=k$ row.
\end{itemize}
We split the factors $(n-j+1)$ in the rows $i=\ell$ 
into \mbox{$(n-j)$} and $(+1)$. The $(+1)$'s aggregate to 
multiplication of $ \cR_{n,m}(x,\lamb_1,\ldots,\lamb_m)$ by 
\mbox{$m(m-1)/m$}.
The modified summation (\ref{eq:difsum})
becomes a sum of $m$ special instances $c_j=n-j$
of the following lemma:
\begin{lemma} \label{th:cmatrix}
For any $m\times m$ matrix 
$\left( \begin{array}{@{}l@{\;}l@{}} a_{i,\,j} & \rangle_{i=1}^{m}  \end{array}  \right)$
and $m$ scalars $c_1,\ldots,c_m$ we have
\begin{align*}
\sum_{\ell=1}^m 
\det \!  \left(   \begin{array}{@{}l@{\;}l@{}}
a_{i,\,j} & \rangle^{i\leq m}_{i\neq \ell} \\
c_j\,a_{i,\,j}  & \rangle_{i=\ell} 
\end{array}  \right)
=\Big( \sum_{j=1}^m c_j \Big) \det \!  \Big(   \begin{array}{@{}l@{\;}l@{}}
a_{i,\,j} & \rangle_{i=1}^{m} 
\end{array}  \Big).
\end{align*}
\end{lemma}
\begin{proof}
Each of the $m!$ expanded terms of $\det \big( a_{i,j} \big)$
gets multiplied by $c_1,\ldots,c_m$ among the $m\cdot m!$ expanded terms
on the  left-hand side.
\end{proof}
In conclusion, (\ref{eq:opxx}) multiplies $\cR_{n,m}(x,\lamb_1,\ldots,\lamb_m)$ by
\begin{equation*}
m-1+\sum_{j=1}^m (n-j) = m-1+mn-{m+1\choose 2}=mn-{m\choose 2}-1.
\end{equation*}

\subsection{Solutions and recurrences}
\label{sec:solrec}

Here we are concerned with solving the differential equations $\PP_{N,M}[y]Y=0$ 
for $M\ge 2$.
\begin{lemma}
Let $N$ denote a positive integer, and let 
\begin{align*}
\LQ_2= &\, y-x+N-1, \\ 
\LQ_3= &\,\LQ_2^2+2x-N+1, \\
\LQ_4= &\,\LQ_2^3+3\,(2x-N+1)\,(\LQ_2-1)-N+1.
\end{align*}
\begin{itemize}
\item A solution of $\PP_{N,N-2}[y]Y=0$ is 
\begin{align} \label{eq:solm2}
Y_{N,2}(x,y) = & \, N\, \hpgg01{N}{xy}+y \,\hpgg01{N+1}{xy} +\LQ_2 \, \exxp{y}
\, \cH_{N+1}^0(y,x). 
\end{align} 
\item A solution of $\PP_{N,N-3}[y]Y=0$ is 
\begin{align} \label{eq:solm3}
Y_{N,3}(x,y) = \, & N\,(\LQ_2-2)\,\hpgg01{N}{xy}
+y\,(\LQ_2-1)\,\hpgg01{N\!+\!1}{xy} \quad \nonumber \\ & 
\!+L_3 \, \exxp{y}\, \cH_{N+1}^0(y,x). 
\end{align}
\item A solution of $\PP_{N,N-4}[y]Y=0$ is 
\begin{align} \label{eq:solm4}
Y_{N,4}(x,y) = \, & \,N\,\big((\LQ_2-2)^2+2y+2x+2\big)\,\hpgg01{N}{xy}  \nonumber \\
& \! +y\,\big((\LQ_2-1)^2+y+3x-N+2\big)\,\hpgg01{N+1}{xy}   \quad \nonumber  \\[1pt]
& \! +L_4 \, \exxp{y} 
\, \cH_{N+1}^0(y,x). 
\end{align}
\end{itemize}
\end{lemma}
\begin{proof}
The operator $\PP_{N,N-2}[y]$ factors as $\LO_2\LO_1$ in 
$\RR(x,N,y)\langle\partial/\partial y\rangle$,
with
\begin{align*}
\LO_1 & = \dif{y}-1-\frc{1}{L_2}, \\
\LO_2 & =  y\diff{y}{2}+\left(N+\frc{y}{L_2}\right) \dif{y}
-x-1+\frc{y+N}{L_2}-\frc{y}{L_2^2}.
\end{align*}
The solution of $\LO_1Y_2=0$ is $Y_2(x,y)=L_2\exxp{y}$. 
This is a solution of  $\PP_{N,N-2}[y]=\LO_2\LO_1Y=0$ as well.
A non-logarithmic solution of $\LO_2Y=0$ is
\[
Y_3(y)=\frc{1}{L_2}\left( \hpgg01{N}{xy}+\frc{y}{N} \,\hpgg01{N+1}{xy} \right).
\]
Solving $\PP_{N,N-2}[y]Y=0$ now means solving the non-homogeneous $\LO_1Y=Y_3$.
This leads to the integration
\[
\int_0^{y} \frac{\exxp{-t}}{(t-x+N-1)^2}
\left( \hpgg01{N}{xt}+\frac{t}{N} \,\hpgg01{N+1}{xt} \right) dt.  
\]  
After a step of integration by parts (and multiplication by $N$), we obtain (\ref{eq:solm2}).

The other two operators $\PP_{N,N-M}[y]$,  $M\in\{3,4\}$ 
factor similarly $\LO^{(M)}_2\LO^{(M)}_1$ as $\PP_{N,N-2}[y]$. 
In the same way, by solving the first order $\LO^{(M)}_1Y^{(M)}_2=0$ and the second order 
$\LO^{(M)}_2Y^{(M)}_3=0$,  we are led to solving the non-homogeneous first order 
$\LO^{(M)}_1Y=Y^{(M)}_3$. 
The equations $\LO^{(M)}_2Y^{(M)}_3=0$ have $M-1$ apparent singularities 
defined by $L_M=0$. {\sf Maple 18} does not solve them,
but looking at (non-logarithmic) power series solutions multiplied by $L_M$ we recognize
the solutions
\begin{equation}
Y_3^{(M)}=\frac{1}{L_M} \, \sum_{j=0}^{M}  \, 
{M \choose j} \,\frac{y^j}{(N-M+1)_{j}}\,\hpgg01{N\!-\!M\!+\!j\!+\!1}{xy}.
\end{equation}
Similarly as for $M=2$, the solutions of $\LO^{(M)}_1Y=Y^{(M)}_3$ are integrals 
that can be simplified to (\ref{eq:solm3}) or (\ref{eq:solm4}) by applying integration 
by parts a few times. 
\end{proof}
A general non-logarithmic solution of $\PP_{N,N-2}[y]$ is 
\begin{equation} \label{eq:gencc}
C_1Y_{N,2}(x,y)+C_2L_2\exxp{y}.
\end{equation}
The function $G_{n,2}(x,y)$ in (\ref{eq:osol2}) differs from $Y_{n,2}(x,y)$ 
by
\begin{align} \label{eq:intconst}
C_2=& \,-\int_0^\infty \! \exxp{-t} \, \hpgg01{n+1}{xt} dt \\
\label{eq:intconst2}
=& -n\,x^{-n}\,\exxp{x}\,\gamma(n,x).
\end{align}
The latter expression is obtained by expanding the $\hpgo01$-series,
and recognizing a $\hpgo11$-series for the incomplete gamma function $\gamma(n,x)$
after the definite integration.   
Similarly, general non-logarithmic solutions of $\PP_{N,N-3}[y]$, $\PP_{N,N-4}[y]$
are obtained by scalar multiplication (by $C_1$) and 
considering the integral $\cH_{N+1}^0(y,x)$ with an integration constant $C_2$.

We observe empirically, especially from differentiation relations between $L_2,L_3,L_4$, 
that 
\begin{equation}
Y_{N,M-1}(x,y)=(M-1) \, \Big(\dif{y}-1\Big) \, Y_{N,M}(x,y)
\end{equation}
for $M=3,4$. This observation indeed generalizes,
leading us to lowering and raising operators on non-logarithmic solutions.
\begin{theorem} \label{th:gsols}
Let $Y_{N,M}(x,y)$ denote a solution of $\PP_{N,N-M}[y]$.
\begin{enumerate}
\item The differential operator $\dif{y}-1$ transforms $Y_{N,M}(x,y)$ 
to a solution of $\PP_{N,N-M+1}[y]$.
\item The differential operator
\begin{equation} \label{eq:gdrec}
y\,\diff{y}2+(N-M+1)\,\dif{y}-x-M
\end{equation}
transforms $Y_{N,M}(x,y)$ to a solution of $\PP_{N,N-M-1}[y]$.
\item If $Y_{N,M-1}(x,y)=\Big(\dif{y}-1\Big)\,Y_{N,M}(x,y)$ \\
and $Y_{N,M-2}(x,y)=\Big(\dif{y}-1\Big)^{\!2} \,Y_{N,M}(x,y)$, then
\begin{align} 
& (y-x+N-2M+1)Y_{N,M}(x,y) \\
& +(2y+N-M+1)Y_{N,M-1}(x,y)+yY_{N,M-2}(x,y)  \qquad \nonumber
\end{align}
is a solution of $\PP_{N,N-M-1}[y]$.
\end{enumerate}
\end{theorem}
\begin{proof}
The first claim follows from the commutation relation
\begin{equation}
\PP_{N,N-M+1}[y]\, \Big(\dif{y}-1\Big)=\Big(\dif{y}-1\Big)\, \PP_{N,N-M}[y].
\end{equation}
The second claim follows by rewriting 
\[ 
\PP_{N,N-M-1}[y] =  
\Big(y\,\diff{y}2+(N\!-\!M\!+\!1)\,\dif{y}-x-M\Big)\!\Big(\dif{y}-1\Big)-M. \quad
\]
The last claim similarly follows from
\begin{align} \PP_{N,N-M-1}[y] = &\, 
y\Big(\dif{y}-1\Big)^3+(2y+N-M+1)\Big(\dif{y}-1\Big)^2 \nonumber \\
&+(y-x+N-2M+1)\Big(\dif{y}-1\Big)-M.
\end{align}
\end{proof}
The recurrence (\ref{eq:grec}) in Conjecture \ref{th:conjecture} is 
a slight modification of the differential operator (\ref{eq:gdrec}), 
and the functions $G_{n,m}(x,y)$ 
differ from the solutions $Y_{n,m}(x,y)$ of this section by 
the difference (\ref{eq:intconst}) and the sign $(-1)^m$.

\section{Holonomic systems for $m\le 3$}
\label{sec:m2}

The results of this article originated from explicit computations for the matrix dimensions 
$m=2$ and $m=3$. The aim was holonomic systems for $\cf_{n,m}(x,\lamb_1,\ldots,\lamb_m)$,
so that the holonomic gradient method \cite{HNTT2013}  could be applied for numeric computation
of the probability density function. 

\subsection{The rank 12 system}

With $m=2$, the holonomic system of Theorem \ref {th:difops} 
has rank 12. It is easy to obtain 
from standard differential equations for 
$\hpgg01{n-1}{x\lamb_1}$, $\hpgg01{n-1}{x\lamb_2}$
and the integrals 
\[
\cH_{n-1}^{n-1}(x,\lamb_1), \quad \cH_{n-1}^{n-1}(x,\lamb_2), \quad
\cH_{n-2}^{n-1}(x,\lamb_1), \quad \cH_{n-2}^{n-1}(x,\lamb_2).
\] 
This was demonstrated by Christoph Koutschan on his {\sf Mathematica} package.
The singularities of the holonomic system are  along  
\begin{equation} \label{eq:sing12}
x=0, \quad \lamb_1=0, \quad \lamb_2=0, \quad  \infty\mbox{-compactifations}.
\end{equation}
It is generated by these three differential operators of order 2 or 3:
\begin{align}
& \lamb_1\diff{\lamb_1}{2}+\lamb_2\diff{\lamb_2}{2}
-(\lamb_1-n+1)\dif{\lamb_1}-(\lamb_2-n+1)\dif{\lamb_2} +x\dif{x}-2n+2,\\
\label{eq:mixed}
& \frc{\partial^3}{\partial x\partial \lamb_1\partial \lamb_2}
+2\frc{\partial^2}{\partial \lamb_1\partial \lamb_2}
-\dif{\lamb_1}-\dif{\lamb_2},\\
& \! \left({\lamb_1\lamb_2}\dif{\lamb_1}
+{\lamb_1\lamb_2}\dif{\lamb_2}
+(n-1)(\lamb_1+\lamb_2)\right) \frc{\partial^2}{\partial \lamb_1\partial \lamb_2} 
+(n-1)x \nonumber \\
& \qquad + \! \left(x\dif{x}+2x-n+2\right) \!\!
\left(x\dif{x}-\lamb_1\dif{\lamb_1}-\lamb_2\dif{\lamb_2}+x-2n+2\right). 
\end{align}
The first equation as in Theorem \ref{th:elimx2}.
For the sake of compactness, the last equation is expressed using non-commutative multiplication 
(in the last term).

Elimination of $\partial/\partial x$ and $\partial/\partial \lamb_2$ 
leads to the fifth order operator
\begin{align}  \label{eq:order5}
\lamb_1^2\diff{\lamb_1}{5}+\lamb_1(2n-\lamb_1+2)\diff{\lamb_1}{4}
+(n^2+n-2x\lamb_1-2n\lamb_1-3\lamb_1)\diff{\lamb_1}{3} \qquad \nonumber \\
-(n^2+2n-2x\lamb_1+2nx)\diff{\lamb_1}{2}+x(2n+x+1)\dif{\lamb_1}-x^2.
\end{align}
Consistent with the theory of Gr\"obner bases, 
the rank 8 system has elimination equations with the leading monomials
$\partial^4/\partial \lamb_1^3\partial \lamb_2+\ldots\;$ and $\;\partial^2/\partial \lamb_2^2+\ldots$.
The fifth order operator 
does not involve the variable $\lamb_2$ even.
It factorizes as the LCLM of $\LL_{n-2}[\lamb_1]$ and $\PP_{n,n-2}[\lamb_1]$. 
Correspondingly, the holonomic system 
factorizes nicely  to a direct sum  of 
two 6-dimensional subspaces, each of those subspaces  is  
a tensor product of a rank 2 system in one variable $\lamb_1$ or $\lamb_2$,
and  rank 3 system in the other variable.
The factorization corresponds nicely with the terms in the expanded determinantal formula
(\ref{th:om2}), as $\hpgg01{n-1}{xy}$ is a solution of $\LL_{n-2}[y]$,
and $G_{n,2}(x,y)$ is a solution of $\PP_{n,n-2}[y]$.

\subsection{Proof of formula (\ref{th:om2})}
\label{sec:solrecp}
 
We seek to prove
\begin{equation} \label{eq:r2det}
\cR_{n,2}(x,\lamb_1,\lamb_2)= 
\frac{x^{2n-2}\,\exxp{-2x}}{n\,(n-1)} \,
\det \left( \begin{array}{cc}
\hpgg01{n-1}{x\lamb_1} & \hpgg01{n-1}{x\lamb_2} \\[2pt]
G_{n,2}(x,\lamb_1)  & G_{n,2}(x,\lamb_2)
\end{array} \right).
\end{equation}
With the same holonomic system of rank 12 established for both sides, 
it is enough to compare a few coefficients in the two series expansions in $\lamb_1,\lamb_2$.
The subspace of non-logarithmic anti-symmetric solutions is 2-dimensional,
hence it is enough to compare 2 pairs of independent coefficients.

After division by $\lamb_1-\lamb_2$ as in (\ref{th:om2}),
a proper general setting is expansion in terms of the symmetric 
{\em Schur polynomials} \cite{WikiWishart}
in $\lamb_1,\lamb_2$. The Schur polynomials
functions are defined in terms of monomial determinants.
Correspondingly, we formulate the following general statement.
\begin{lemma} \label{th:schur}
Consider $m$ functions $f_1(y),\ldots,f_m(y)$ defined by the convergent series
\begin{equation} \label{eq:fff}
f_k(y)=\sum_{k=0}^{\infty} c^{(k)}_j y^j.
\end{equation}
Then
\begin{equation} \label{eq:detdets}
\det \big( \, f_i(\lamb_j) \; \rangle_{i=1}^m  \, \big) =
\mathop{ \sum \, \cdots \, \sum }_{0\le q_1<q_2<\ldots<q_m}
\det \Big( \, c^{(i)}_{q_j} \;\, \big\rangle_{j=1}^m  \, \Big) \,
\det \Big( \, \lamb_i^{q_j} \; \big\rangle_{j=1}^m  \, \Big).
\end{equation}
\end{lemma}
\begin{proof}
Intermediate expansions are
\begin{align}
= & \sum_{q_1=0}^\infty \, \cdots \, \sum_{q_m=0}^{\infty}
\det \Big( \, c^{(i)}_{q_j} \;\, \big\rangle_{j=1}^m  \, \Big) \,\lamb_1^{q_1}\cdots \lamb_m^{q_m} \\
= & \sum_{q_1=0}^\infty \, \cdots \, \sum_{q_m=0}^{\infty}
 c^{(1)}_{q_1} \cdots c^{(m)}_{q_m} \,
\det \Big( \, \lamb_i^{q_j} \; \big\rangle_{j=1}^m  \, \Big).
\end{align}
The newest determinants with some $q_i=q_j$ for $i\neq j$ are zero.
After collecting the terms with the same sets $\{q_1,\ldots,q_m\}$,  
we get the result.
\end{proof}

We first apply this lemma to the determinant in (\ref{eq:start}) 
with $f_1(y)=\cH_{n-1}^{n-1}(x,y)$,  $f_2(y)=\cH_{n-1}^{n-2}(x,y)$.
Therefore 
\[
c^{(1)}_j=\frac{\gamma(n+j,x)}{(n-1)_j\;j!}, \qquad c^{(2)}_j=\frac{\gamma(n+j-1,x)}{(n-1)_j\;j!}.
\]
We differentiate (\ref{eq:detdets}) to get
\begin{align} \label{eq:start2}
\cR_{n,2}(x,\lamb_1,\lamb_2) & =  
\sum_{i=0}^{\infty} \sum_{j=i+1}^{\infty} \!
\frac{ {\displaystyle \frac{d}{d x}} \det \! \left( \! 
\begin{array}{cc} \gamma(n+i,x) & \! \gamma(n+i-1,x) \\
\gamma(n+j,x) & \! \gamma(n+j-1,x) \end{array}
\! \right)}{(n-1)_i(n-1)_j\,i!j!}
\det \! \left( \! \begin{array}{cc}
\!\lamb_1^i & \!\! \lamb_2^i \! \\[2pt]
\! \lamb_1^j & \!\! \lamb_2^j \! \\
\end{array} \! \right) \nonumber \\ 
& = \frac{d}{d x} \det \left( \! 
\begin{array}{cc} \gamma(n,x) & \gamma(n-1,x) \\
\gamma(n+1,x) & \gamma(n,x) \end{array}
\! \right)\frac{\lamb_2-\lamb_1}{n-1}  \\
& \quad + \frac{d}{d x} \det \left( \! 
\begin{array}{cc} \gamma(n,x) & \gamma(n-1,x) \\
\gamma(n+2,x) & \gamma(n+1,x) \end{array}
\! \right)\frac{\lamb^2_2-\lamb^2_1}{2n(n-1)}  \nonumber \\
& \quad + \ldots  \nonumber \\
& \hspace{-44pt} = 
\left( \! \Big( \frac{x^2}{n-1}-2x+n \Big) \gamma(n-1,x)+\frac{x-n}{n-1}\,x^{n-1}\exxp{-x} \right)
x^{n-2}\exxp{-x} \,   (\lamb_1-\lamb_2)  \nonumber \\
\label{eq:detse2} &\hspace{-30pt} \! +  
\left( \! \Big( \frac{x^3}{n(n-1)}-\frac{x^2}n-x+n+1 \Big) \gamma(n-1,x) \right.\\
 & \left. \hspace{38pt}  +\frac{(x-n-1)(x+n)}{n(n-1)}\,x^{n-1}\exxp{-x} \right) 
   x^{n-2}\exxp{-x}\,\frac{\lamb^2_1-\lamb^2_2}2  \nonumber \\
& \hspace{-30pt} + \ldots.  \nonumber
\end{align}
Surely, recurrence (\ref{incgamma}) has been used. 
Considering the left-hand side of (\ref{eq:r2det}), we set $c^{(1)}_j=x^j/\big(j!\,(n-1)_j\big)$. 
An expansion of (\ref{eq:solm2}) with $N=n$ is
\begin{align}
Y_{n,2}(x,y) = & \,  \, n+ny
+\sum_{k=0}^\infty \left( n+\sum_{j=0}^k \frac{k+1-j}{(n+1)_j} \, x^j\right) \frac{y^{k+2}}{(k+2)!}.
\end{align}
We only need the first few terms.
To get coefficients of $G_{n,2}(x,y)$, we add the difference $C_2(y-x+n-1)\exxp{y}$
as in (\ref{eq:intconst})--(\ref{eq:intconst2}),
thus adding
\begin{align*}
n\,x^{-n}\,\exxp{x}\,\gamma(n,x)\sum_{k=0}^{\infty} (x-n-k+1)\frac{y^k}{k!}.
\end{align*}
Therefore
\begin{align} \label{eq:ccc}
c^{(2)}_0= & \, n+n\,(x-n+1)\,x^{-n}\exxp{x}\gamma(n,x), \nonumber \\
c^{(2)}_1= & \, n+n\,(x-n)\,x^{-n}\exxp{x}\gamma(n,x), \\
c^{(2)}_2= & \, \frac{n+1}2+\frac{n}2\,(x-n-1)\,x^{-n}\exxp{x}\gamma(n,x). \nonumber
\end{align}
The determinant in  (\ref{eq:r2det}) expands as
\begin{align*}
 \det \left( \! 
\begin{array}{cc}1 & \frac{x}{n-1} \\
c^{(2)}_0 & c^{(2)}_1 \end{array} \right) (\lamb_2-\lamb_1)
+ \det \left( \! 
\begin{array}{cc}1 &  \frac{x^2}{2n(n-1)} \\
c^{(2)}_0 & c^{(2)}_2 \end{array} \right) (\lamb_2^2-\lamb_1^2)
\end{align*}
We get the same two terms as in (\ref{eq:detse2}).

\subsection{The rank 8 system}
\label{sec:rank8}

A rank 8 holonomic system for $\cR_{n,2}(x,\lamb_1,\lamb_2)$ 
is obtained by expressing this function and its derivatives 
as linear combinations of
\begin{align*}
x^{n}\,\exxp{-x}\hpgg01{n-1}{x\lamb_2} \cH_{n-1}^{n}(x,\lamb_1), && 
x^{n}\,\exxp{-x}\hpgg01{n}{x\lamb_2} \cH_{n-1}^{n}(x,\lamb_1), \\
x^{n}\,\exxp{-x}\hpgg01{n-1}{x\lamb_1} \cH_{n-1}^{n}(x,\lamb_2), && 
x^{n}\,\exxp{-x}\hpgg01{n}{x\lamb_1} \cH_{n-1}^{n}(x,\lamb_2), \\
x^{2n}\,\exxp{-2x}\hpgg01{n-1}{x\lamb_1}\hpgg01{n-1}{x\lamb_2}, && 
x^{2n}\,\exxp{-2x}\hpgg01{n-1}{x\lamb_1}\hpgg01{n}{x\lamb_2}, \\
x^{2n}\,\exxp{-2x}\hpgg01{n}{x\lamb_1}\hpgg01{n-1}{x\lamb_2}, && 
x^{2n}\,\exxp{-2x}\hpgg01{n}{x\lamb_1}\hpgg01{n}{x\lamb_2}.
\end{align*}
These expressions follow from Theorem \ref{th:hrank} \refpart{iv}. 
For example, the expression of $\cR_{n,2}(x,\lamb_1,\lamb_2)$ has these coefficients, respectively:
\begin{align*}
\frac{\lamb_1-x}{n-1}+1, \quad 0, \quad 
-\frac{\lamb_2-x}{n-1}-1, \quad 0, \quad
0, \quad \frac{x}{n-1}-1, \quad -\frac{x}{n-1}+1, \quad 0.
\end{align*}
Further, the expression of $(n-1)\,\partial \cR_{n,2}/\partial x$ has these coefficients:
\begin{align}
-1, \qquad \lamb_2\,\Big(\frac{\lamb_1-x}{n-1}+1\Big), \qquad 
1, \qquad -\lamb_1\,\Big(\frac{\lamb_2-x}{n-1}+1\Big), \\
0, \qquad -\lamb_2+1, \qquad \lamb_1-1, \qquad 
(\lamb_1-\lamb_2)\,\Big(\frac{x}{n-1}-1\Big); \nonumber
\end{align}
and so on. The rank 8 system has singularities not just along (\ref{eq:sing12}), 
but additionally along \mbox{$\lamb_1=\lamb_2$} and the hypersurface $S_n(x,y)=0$, where
\begin{equation}
S_n(x,y)=(2x-n+1)^2\,(\lamb_1+\lamb_2-2x+2n-2)-\frac{x}{2}\,(\lamb_1-\lamb_2)^2.
\end{equation}
It contains another second order operator
\begin{align} \!
& S_n(x,y)  \Big(  x\DD_x^2 
 +\frac{2\lamb_1\lamb_2}{x} \frc{\partial^2}{\partial \lamb_1\partial \lamb_2}
 \! -\DD_\lamb \left(2\DD_x+1\right)+(x-n+1)\DD_x+\lamb_1+\lamb_2-1\Big) \nonumber \\
& +(\lamb_1+\lamb_2-4x+2n-2) \, \Big( (x-n+1) \! \Big( \DD_\lamb \DD_x
-\frac{2\lamb_1\lamb_2}{x}\frac{\partial^2}{\partial \lamb_1\partial \lamb_2}\Big) \\ 
& \hspace{122pt} 
 +(2x-n) \! \left( \DD_\lamb+(x-n+1)\DD_x -1 \right) \Big)   \hspace{56pt}  \nonumber \\
& +2 (2x-n+1)  \Big( (\lamb_1+\lamb_2) \left( \DD_\lamb+2x-n+1\right)
-\lamb_1^2\frac{\partial}{\partial \lamb_1^2}-\lamb_2^2\frac{\partial}{\partial \lamb_2^2}
-\frac{(\lamb_1+\lamb_2)^2}2 \Big)  \nonumber \\
& +2 \lamb_1\lamb_2  \Big( x\DD_\lamb\DD_x-\frac{\lamb_1+\lamb_2}2-2x+n \Big) \nonumber\\
& \hspace{122pt}   +(\lamb_1-\lamb_2) \!
 \Big( \lamb_1^2\frac{\partial^2}{\partial \lamb_1^2}-\lamb_2^2\frac{\partial^2}{\partial \lamb_2^2}
 -\lamb_1^2\frac{\partial}{\partial \lamb_1}+\lamb_2^2\frac{\partial}{\partial \lamb_2} \Big),
  \nonumber
\end{align}
where
\begin{equation}
\DD_x=\dif{x}+1-\frac{n-2}{x},  \qquad
\DD_\lamb=\lamb_1\dif{\lamb_1}+\lamb_2\dif{\lamb_2}.
\end{equation}
The smaller rank system appears to be more complex.
It can be obtained from the rank 12 system by adjoining this 3rd order operator:
\begin{align}
\Big(2\lamb_1\dif{\lamb_1}+2\lamb_2\dif{\lamb_2}
-3\lamb_1-3\lamb_2+4n-6\Big)  \frc{\partial^2}{\partial \lamb_1\partial \lamb_2} 
+\lamb_2\dif{\lamb_1}+\lamb_1\dif{\lamb_2}  \nonumber \\
-\Big(x\diff{x}{2}+(x-n+3)\dif{x}+n\Big)
\! \Big(\dif{\lamb_1}+\dif{\lamb_2}\Big)+3x \dif{x}-2n+6.
\end{align}

Elimination of $\partial/\partial x$ and $\partial/\partial \lamb_2$ leads to the same fifth order operator
(\ref{eq:order5}). The elimination equations with the leading monomials
$\partial^4/\partial \lamb_1^2\partial^2 \lamb_2+\ldots\;$ and $\;\partial^2/\partial \lamb_2^3+\ldots$.
Factorization of the solution spaces of operator (\ref{eq:order5}) 
is harder to follow, as the 6-dimensional subspaces intersect.

\subsection{The case $m=3$}
\label{sec:m3}

The holonomic system of Theorem \ref{th:difops} has rank 108 when $m=3$.
A Gr\"obner basis computation without $\partial/\partial x$ is fast on {\sf Maple 18}
(with respect to a total degree ordering in $\partial/\partial\lamb_k$'s, 
on a 2.8GHz {\sf MacBook Pro} of 2014). 
The lowest order operator in the $\partial/\partial\lamb_k$'s is of order 5.
Allowing $\partial/\partial x$, it is equivalent to 
\begin{equation}
\frc{\partial^4}{\partial x\partial \lamb_1\partial \lamb_2\partial \lamb_3}
+3\frc{\partial^3}{\partial \lamb_1\partial \lamb_2\partial \lamb_3}
-\frc{\partial^2}{\partial \lamb_1\partial \lamb_2}
-\frc{\partial^2}{\partial \lamb_1\partial \lamb_3}
-\frc{\partial^2}{\partial \lamb_2\partial \lamb_3}.
\end{equation}
This expression is comparable with (\ref{eq:mixed}). 
A Gr\"obner basis computation with $\partial/\partial x$
leads to rapid increase of memory usage, 4GB in a few minutes.

Replacing $\TT_2[y]$ by (\ref{eq:altp3}) gives the same LCLM operators,
but a different holonomic system of rank 108.
A similar Gr\"obner basis computation without $\partial/\partial x$ takes about 8 minutes.
The lowest order operator is 
\begin{align}
\sum_{k=1}^3 \Big( & \lamb_k^2\diff{\lamb_k}4+(2n-2-\lamb_k)\diff{\lamb_k}3
+(n^2-3n+2-(x+2n)\lamb_k)\diff{\lamb_k}2  \nonumber \\
&  +(x\lamb_k-(n-2)(x+n+1))\dif{\lamb_k} \Big)+(3n-2)x.
\end{align}
Combining both holonomic systems leads to formidable Gr\"obner basis computation,
apparently. Computation of differential operators for the rank $\le 26$ system 
of Theorem \ref{th:hrank}\refpart{iv} is barely viable on {\sf Singular 4} 
(given several hours), but further manipulation is hard.

Conjecture \ref{th:conjecture} was checked for $m=4$
by expanding both sides of (\ref{eq:conj}) in the determinants of
\begin{equation*}
\left( \! \begin{array}{ccc}
1 & 1 & 1 \\
\lamb_1 & \! \lamb_2 &  \! \lamb_3 \\
\lamb_1^2 & \! \lamb_2^2 &  \! \lamb_3^2 
\end{array} \! \right), \
\left( \! \begin{array}{ccc}
1 & 1 & 1 \\
\lamb_1 & \! \lamb_2 &  \! \lamb_3 \\
\lamb_1^3 & \! \lamb_2^3 &  \! \lamb_3^3 
\end{array} \! \right), \
 \left( \! \begin{array}{ccc}
1 & 1 & 1 \\
\lamb_1^2 & \! \lamb_2^2 &  \! \lamb_3^2 \\
\lamb_1^3 & \! \lamb_2^3 &  \! \lamb_3^3 
\end{array} \! \right), \
 \left( \! \begin{array}{ccc}
\lamb_1 & \! \lamb_2 &  \! \lamb_3 \\
\lamb_1^2 & \! \lamb_2^2 &  \! \lamb_3^2 \\
\lamb_1^3 & \! \lamb_2^3 &  \! \lamb_3^3 
\end{array} \! \right), \ldots,
\end{equation*}
and comparing the coefficients to these four determinants.
For example, comparison of the first coefficient by Lemma \ref{th:schur} gives
\begin{align*}
\frac{1}{2(n\!-\!2)^2(n\!-\!1)}\dif{x} \det & \left( \begin{array}{ccc}
\gamma(n,x) & \gamma(n-1,x) & \gamma(n-2,x) \! \\
\! \gamma(n+1,x) & \gamma(n,x) & \gamma(n-1,x) \! \\
\! \gamma(n+2,x) & \gamma(n+1,x) & \gamma(n,x) \\
\end{array} \right) \\
& \hspace{60pt} = C(x) \det \left( \! 
\begin{array}{ccc}1 & \frac{x}{n-2} & \! \frac{x^2}{2(n-2)(n-1)} \!\! \\[3pt]
\widetilde{c}^{\,(2)}_0 & \widetilde{c}^{\,(2)}_1 & \widetilde{c}^{\,(2)}_2  \\[2pt]
c^{(3)}_0 & c^{(3)}_1 & c^{(3)}_2 \end{array} \right)\!,
\end{align*}
where $\widetilde{c}^{\,(2)}_k$ are the shifted $n\mapsto n-1$ versions of $c^{(2)}_k$ in (\ref{eq:ccc}),
and $c^{(3)}_k$ are the first coefficients of the expansion of $G_{n,3}(x,y)$:
\begin{align}
c^{(3)}_0= & \, -nx+n(n-3)-n\, \big((x-n)^2+4x-3n+2\big) x^{-n}\exxp{x}\gamma(n,x), \nonumber \\
c^{(3)}_1= & \, -nx+n(n-1)-n\, \big((x-n)^2+2x-n\big) x^{-n}\exxp{x}\gamma(n,x), \\
c^{(3)}_2= & \, \frac{-nx+n(n+1)}2-\frac{n}2\,\big((x-n)^2+n\big) x^{-n}\,\exxp{x}\gamma(n,x). \nonumber
\end{align}
For an intermediate check, here is a quadratic expression 
in $A=\gamma(n-1,x)$ and $E=x^{n-1}\exxp{-x}$
for the first coefficient:
\begin{align}
\frac{x^{n-3}\exxp{-x}}{n-2} \,\Big( & 
-\!\big( (x-n+1)^2 A+(x-n)E \big)^2  +(n-1)(n-1-4x) A^2 \quad \nonumber \\
&+2(x^2+2x-n^2+n)AE +2(x+n)E^2  \Big).
\end{align}
To set up $C(x)$ in the conjecture, we compared the similar first coefficients for $m=4$ as well.

\subsection*{Acknowledgement}

This work was supported by the JSPS Grant-in-Aid for Scientific
Research No.~25220001.

\small

\bibliographystyle{abbrv}
\bibliography{LRWishart}

\end{document}